\documentclass{amsart}

\usepackage[latin1]{inputenc}
\usepackage[english]{babel}
\usepackage{indentfirst}
\usepackage{amssymb}
\usepackage{amsthm}
\usepackage{xcolor}
\usepackage[all]{xy}
\usepackage[mathscr]{eucal}

\newcommand{\impli}{\Rightarrow}
\newcommand{\Nat}{\mathbb{N}}
\newcommand{\N}{\mathbb{N}}

\newcommand{\sub}{\subseteq}
\def\epsilon{\varepsilon}
\newtheorem{theo}{Theorem}
\newtheorem{lem}[theo]{Lemma}
\newtheorem{sublem}[theo]{Sublemma}
\newtheorem{cor}[theo]{Corollary}
\newtheorem{defi}[theo]{Definition}

\newtheorem{exa}[theo]{Example}

\title[A.e. convergent sequences of weak$^*$-to-norm continuous operators]{Almost everywhere convergent sequences of weak$^*$-to-norm continuous operators}

\author{Jos\'{e} Rodr\'{i}guez}
\address{Dpto. de Ingenier\'{i}a y Tecnolog\'{i}a de Computadores\\Facultad de Inform\'{a}tica\\
Universidad de Murcia\\ 30100 Espinardo (Murcia)\\ Spain} \email{joserr@um.es}
\subjclass[2010]{46B28, 46G10}

\keywords{Weak$^*$-to-norm continuous operator; bounded approximation property; Asplund space; Banach space not containing~$\ell_1$; Bourgain property}

\thanks{Research partially supported by {\em Agencia Estatal de Investigaci\'{o}n} [MTM2017-86182-P, grant cofunded by ERDF, EU] 
and {\em Fundaci\'on S\'eneca} [20797/PI/18]}

\begin{document}

\begin{abstract}
Let $X$ and $Y$ be Banach spaces, and $T:X^*\to Y$ be an operator. We prove that if $X$
is Asplund and $Y$ has the approximation property, then for each Radon probability~$\mu$ on~$(B_{X^*},w^*)$ there is
a sequence of $w^*$-to-norm continuous operators $T_n:X^*\to Y$ such that $\|T_n(x^*)-T(x^*)\| \to 0$ for $\mu$-a.e. $x^*\in B_{X^*}$;
if $Y$ has the $\lambda$-bounded approximation property for some $\lambda\geq 1$, then the sequence can be chosen in such a way
that $\|T_n\|\leq \lambda\|T\|$ for all $n\in \N$. The same conclusions 
hold if $X$ contains no subspace isomorphic to~$\ell_1$, $Y$ has the approximation property (resp., $\lambda$-bounded approximation property)
and $T$ has separable range. This extends to the non-separable setting
a result by Mercourakis and Stamati.
\end{abstract}

\maketitle

\section{Introduction}

The celebrated Odell-Rosenthal theorem~\cite{ode-ros} (cf., \cite[Theorem~4.1]{van}) states that every separable Banach space $X$ not
containing subspaces isomorphic to~$\ell_1$ is $w^*$-sequentially dense in its bidual, that is, for every $x^{**}\in X^{**}$
there is a sequence $(x_n)$ in~$X$ which $w^*$-converges to~$x^{**}$; moreover, the sequence can be chosen
in such a way that $\|x_n\|\leq \|x^{**}\|$ for all $n\in \N$. Note that, for an arbitrary Banach space~$X$, 
a functional $x^{**}\in X^{**}$ is $w^*$-continuous if and only if it belongs to~$X$. So, it is natural
to ask about extensions of the Odell-Rosenthal theorem for operators (i.e., linear and continuous maps)
in the following sense: if $X$ and $Y$ are Banach spaces, and
$T:X^{*}\to Y$ is an operator, which conditions do ensure the existence of a sequence of 
$w^*$-to-norm continuous operators $T_n:X^*\to Y$ such that $\|T_n(x^*)-T(x^*)\|\to 0$ 
for every $x^*\in X^*$? This type of question was addressed by Mercourakis and Stamati~\cite{mer-sta}, and later by Kalenda and Spurn\'{y}~\cite{kal-spu}.

Throughout this paper $X$ and $Y$ are Banach spaces. We write $\mathcal{L}(X^*,Y)$ to denote the Banach space of all operators from~$X^*$ to~$Y$, equipped
with the operator norm. The linear subspace of all $w^*$-to-norm continuous operators
from~$X^*$ to~$Y$ is denoted by~$\mathcal{F}_{w^*}(X^*,Y)$. Every 
element of $\mathcal{F}_{w^*}(X^*,Y)$ has finite rank and, in fact, is of the form $x^*\mapsto \sum_{i=1}^n x^*(x_i)y_i$ 
for some $n\in \N$, $x_i\in X$ and $y_i\in Y$.

The following result can be found in \cite[Theorem~2.19]{mer-sta} and is the starting point of this paper:

\begin{theo}[Mercourakis-Stamati]\label{theo:MS}
Suppose that $X^*$ is separable and that $Y$ is separable and has the bounded approximation property. Let $T\in \mathcal{L}(X^*,Y)$. Then  
there is a sequence $(T_n)$ in $\mathcal{F}_{w^*}(X^*,Y)$ such that 
\begin{equation}\label{eqn:SOT}
	\|T_n(x^*)-T(x^*)\| \to 0
	\quad\mbox{for every }x^*\in X^*.
\end{equation} 
\end{theo}

We stress that the proof of this result given in~\cite{mer-sta} contains a gap which was corrected in \cite[Remark~4.4]{kal-spu}. 
There are examples showing that the conclusion of Theorem~\ref{theo:MS}
can fail if $Y$ does not have the bounded approximation property (see \cite[Example~2.22]{mer-sta} and \cite[Example~2.3]{kal-spu}), or if
the assumption on~$X$ is weakened to being separable without subspaces isomorphic to~$\ell_1$ (see \cite[Theorem~2.30]{mer-sta}).

In this paper we try to extend Theorem~\ref{theo:MS} to the non-separable setting. As 
we will explain below, the separability assumption on~$Y$ can be removed from the statement of Theorem~\ref{theo:MS}.
On the other hand, the result does not work for an arbitrary Asplund space~$X$, since there are
Asplund spaces which are not $w^*$-sequentially dense in its bidual (e.g., $c_0(\Gamma)$ for any uncountable set~$\Gamma$). 
However, any Asplund space~$X$ satisfies 
a weaker property which can be extended succesfully to
the setting of operators, namely: 
\begin{enumerate}
\item[(*)] for each $x^{**}\in B_{X^{**}}$ and each Radon probability $\mu$ on~$(B_{X^*},w^*)$,
there is a sequence $(x_n)$ in~$B_{X}$ such that $\langle x_n,x^* \rangle \to \langle x^{**},x^*\rangle$ for $\mu$-a.e. $x^*\in B_{X^*}$.
\end{enumerate}
Actually, this property characterizes the non-containment of~$\ell_1$, as a consequence
of Haydon's result~\cite{hay-J} that a Banach space~$X$ contains no subspace isomorphic to~$\ell_1$ if and only if
the identity map $i:(B_{X^*},w^*) \to X^*$ is universally Pettis integrable 
if and only if $i$ is universally scalarly measurable (cf., \cite[Theorems~6.8 and~6.9]{van}).

Our main result reads as follows:

\begin{theo}\label{theo:Main}
Let $T\in \mathcal{L}(X^*,Y)$. Suppose that $Y$ has the approximation property and that one of the following conditions holds:
\begin{enumerate}
\item[(i)] $X$ is Asplund;
\item[(ii)] $X$ contains no subspace isomorphic to~$\ell_1$ and $T$ has separable range.
\end{enumerate} 
Let $\mu$ be a Radon probability on~$(B_{X^*},w^*)$. Then there is a sequence $(T_n)$ in $\mathcal{F}_{w^*}(X^*,Y)$ such that 
\begin{equation}\label{eqn:Asplund}
	\|T_n(x^*)-T(x^*)\| \to 0
	\quad\mbox{for }\mu\mbox{-a.e. }x^*\in B_{X^*}.
\end{equation}
If $Y$ has the $\lambda$-bounded approximation property for some $\lambda\geq 1$, then the sequence
can be chosen in such a way that $\|T_n\|\leq \lambda\|T\|$ for all $n\in \N$.
\end{theo}

To see why Theorem~\ref{theo:Main} implies Theorem~\ref{theo:MS}, observe 
that if $X^*$ is separable and $(x_k^*)$ is a norm-dense sequence in~$B_{X^*}$, then
the formula 
$$
	\mu(A):=\sum_{x_k^*\in A}2^{-k}, \qquad A \in {\rm Borel}(B_{X^*},w^*),
$$ 
defines a Radon probability on $(B_{X^*},w^*)$ such that 
conditions~\eqref{eqn:SOT} and~\eqref{eqn:Asplund} are equivalent
for any norm-bounded sequence $(T_n)$ in~$\mathcal{L}(X^*,Y)$ and any $T\in \mathcal{L}(X^*,Y)$. Note that this argument does
not require that $Y$ is separable, hence the conclusion of Theorem~\ref{theo:MS} still holds
if the separability assumption on~$Y$ is dropped. Anyway, an operator which can be approximated as in 
Theorem~\ref{theo:MS} has separable range.

The proof of Theorem~\ref{theo:Main} is given in Section~\ref{section:proofs} as a consequence of a more general result, 
see Theorem~\ref{theo:nol1-ap}. The difference between cases (i) and~(ii) of Theorem~\ref{theo:Main}
lies in Schwartz's theorem that a Banach space~$X$ is Asplund
if and only if the identity map $i:(B_{X^*},w^*)\to X^*$ is universally strongly measurable
(see, e.g., \cite[Corollary 7.8.7(a)]{bou-J}). As we will show in Example~\ref{exa:JT}, the conclusion of Theorem~\ref{theo:Main} can fail 
for an arbitrary $T\in \mathcal{L}(X^*,Y)$ when $X$ contains no subspace isomorphic to~$\ell_1$ but it is not Asplund.

\subsubsection*{Notation and terminology}

Throughout this paper all topological spaces are Hausdorff and all Banach spaces are real. 

Given a non-empty set $\Gamma$ and a topological space~$S$, we denote by $S^\Gamma$
the set of all functions from~$\Gamma$ to~$S$, equipped with the product topology
$\mathfrak{T}_p$, i.e., the topology of pointwise convergence on~$\Gamma$.

By an {\em operator} we mean a continuous linear map between Banach spaces.
By a {\em subspace} of a Banach space we mean a norm-closed linear subspace.
The topological dual of a Banach space~$X$ is denoted by~$X^*$ and
we write $w^*$ for its weak$^*$-topology. The evaluation of a functional $x^*\in X^*$
at $x\in X$ is denoted by either $\langle x,x^*\rangle$ or $\langle x^*,x\rangle$.
The norm of~$X$ is denoted by $\|\cdot\|$ or $\|\cdot\|_X$.
We write $B_X$ to denote the closed unit ball of~$X$, i.e., $B_{X}=\{x\in X:\|x\|\leq 1\}$. 
Recall that $X$ is said to be {\em Asplund} if every separable subspace of~$X$ has separable dual or, equivalently,
$X^*$ has the Radon-Nikod\'{y}m property (see, e.g., \cite[p.~198]{die-uhl-J}).

A Banach space $Y$ is said to have the {\em approximation property} if for each norm-compact set $C \sub Y$ and each $\epsilon>0$
there is a finite rank operator $R:Y \to Y$ such that $\|R(y)-y\|\leq \epsilon$ for all $y\in C$; if $R$ can be chosen in such a way that
$\|R\| \leq \lambda$ for some constant $\lambda\geq 1$ (independent of~$C$ and~$\epsilon$), then $Y$ is said to have
the {\em $\lambda$-bounded approximation property}. A Banach space has the
{\em bounded approximation property} if it has the $\lambda$-bounded approximation property for some $\lambda\geq 1$. 

The set of all Radon probabilities on~$(B_{X^*},w^*)$ is denoted by $P(B_{X^*})$;
that is, an element $\mu\in P(B_{X^*})$ is a probability measure defined
on ${\rm Borel}(B_{X^*},w^*)$ (the Borel $\sigma$-algebra of $(B_{X^*},w^*)$)
which is inner regular with respect to $w^*$-compact sets, in the sense that 
$\mu(A)=\sup\{\mu(K):\, K \sub A, \, \text{$K$ is $w^*$-compact}\}$ for every $A\in {\rm Borel}(B_{X^*},w^*)$.

\section{Results}\label{section:proofs}

\subsection{Strongly measurable functions and operators}
This subsection contains some auxiliary measure-theoretic results.

\begin{defi}\label{defi:conditions}
Let $(\Omega,\Sigma,\mu)$ be a probability space and $f: \Omega\to Y$ be a function. We say that $f$ 
\begin{enumerate}
\item[(i)] is {\em simple} if it can be written as $f=\sum_{i=1}^ my_i\chi_{A_i}$, where $m\in \Nat$, $y_i\in Y$ and 
$\chi_{A_i}$ is the characteristic function of~$A_i\in \Sigma$;
\item[(ii)] is {\em strongly $\mu$-measurable} if there is a sequence $f_n:\Omega\to Y$ of simple functions such that
$\|f_n(\omega)- f(\omega)\|\to 0$ for $\mu$-a.e. $\omega \in \Omega$;
\item[(iii)] has {\em $\mu$-essentially separable range} if there is $B\in \Sigma$
with $\mu(B)=1$ such that $f(B)$ is separable.
\end{enumerate}
\end{defi}

Clearly, every strongly $\mu$-measurable function has $\mu$-essentially separable range. In fact,
we have the following stronger statement which belongs to the folklore. 

\begin{lem}\label{lem:tight}
Let $(\Omega,\Sigma,\mu)$ be a probability space and $f: \Omega\to Y$ be a strongly $\mu$-measurable function.
Then for every $\epsilon>0$ there is $U\in \Sigma$ such that $\mu(U)\geq 1-\epsilon$
and $f(U)$ is relatively norm-compact.
\end{lem}
\begin{proof} There exist $B\in \Sigma$ with $\mu(B)=1$
and a sequence  $f_n:\Omega\to Y$ of simple functions such that
$\|f_n(\omega)- f(\omega)\|\to 0$ for every $\omega \in B$. 
Let $\Sigma_B:=\{A\cap B:A\in \Sigma\}$ be the trace $\sigma$-algebra on~$B$.
Then the restriction $f|_B$ is $\Sigma_B$-${\rm Borel}(Y,{\rm norm})$-measurable
(see, e.g., \cite[Proposition~8.1.10]{coh-J}). Since $f(B)$ is separable, its norm-closure 
$T:=\overline{f(B)}$, equipped with the norm topology,
is a Polish space. Define a probability $\nu$ on ${\rm Borel}(T)$ by 
the formula $\nu(C):=\mu((f|_B)^{-1}(C))$ for all $C\in {\rm Borel}(T)$.
The fact that $T$ is Polish implies that $\nu$ is regular (see, e.g., \cite[Proposition~8.1.12]{coh-J})
and so we can find a norm-compact set $K \sub T$ such that $\nu(K)\geq 1-\epsilon$. 
The set $U:= (f|_B)^{-1}(K)\in \Sigma_B$ satisfies the required properties. 
\end{proof}

The following lemma generalizes the well known fact that strong $\mu$-measurability is preserved by 
$\mu$-a.e. limits of sequences. 

\begin{lem}\label{lem:double}
Let $(\Omega,\Sigma,\mu)$ be a probability space.
For each $k,n\in \mathbb{N}$, let $f_{k,n}:\Omega\to Y$ be a strongly $\mu$-measurable function.
Suppose that:
\begin{enumerate}
\item[(i)] for each $n\in \mathbb{N}$ there is a function $f_n:\Omega \to Y$
such that 
$$
	\lim_{k\to \infty}\|f_{k,n}(\omega)- f_n(\omega)\|= 0
	\quad\text{for $\mu$-a.e. $\omega\in \Omega$;}
$$
\item[(ii)] there is a function $f:\Omega \to Y$ such that
$$
	\lim_{n\to\infty}\|f_n(\omega)-f(\omega)\|= 0
	\quad\text{for $\mu$-a.e. $\omega\in \Omega$.}
$$
\end{enumerate} 
Then $f$ and each $f_n$ are strongly $\mu$-measurable and there is a sequence $(k_n)$ in~$\mathbb{N}$ such that 
$$
	\lim_{n\to\infty}\|f_{k_n,n}(\omega)-f(\omega)\|=0
	\quad\text{for $\mu$-a.e. $\omega\in \Omega$.}
$$
\end{lem}
\begin{proof} We begin by proving a particular case.

{\em Particular case.} Assume first that each $f_{k,n}$ is simple. Fix $n\in \N$. By~(i), $f_n$ is strongly $\mu$-measurable.
Therefore, for every $k\in \N$ the diffference $f_{k,n}-f_n$ is strongly $\mu$-measurable and so the real-valued
function $\|f_{k,n}(\cdot)- f_n(\cdot)\|$ is $\mu$-measurable.
Egorov's theorem applied to the sequence 
$(\|f_{k,n}(\cdot)- f_n(\cdot)\|)_{k\in \N}$ ensures the existence of~$U_n\in \Sigma$ with $\mu(\Omega \setminus U_n)\leq 2^{-n}$
such that $\|f_{k,n}(\cdot)- f_n(\cdot)\|\to 0$ uniformly on~$U_n$ as $k\to \infty$; in particular, we can choose $k_n\in \N$ such that 
\begin{equation}\label{eqn:double}
	\sup_{\omega\in U_n}\|f_{k_n,n}(\omega)-f_n(\omega)\|\leq \frac{1}{n}.
\end{equation}

Define $U:=\bigcup_{m\in\N}\bigcap_{n\geq m}U_n \in \Sigma$, so that
$\mu(U)=1$. Now, we use~(ii) to pick $V\in \Sigma$ with $\mu(V)=1$ such that $\|f_n(\omega)-f(\omega)\|\to 0$ for every $\omega\in V$.
Then $\mu(U\cap V)=1$ and, bearing in mind~\eqref{eqn:double}, we deduce that 
$$
	\lim_{n\to\infty}\|f_{k_n,n}(\omega)-f(\omega)\|=0
	\quad\text{for every $\omega\in U\cap V$}.
$$
In particular, $f$ is strongly $\mu$-measurable. This finishes the proof of the {\em Particular case}. Note that, as a consequence, 
we also deduce the following fact.

{\sc Fact.} {\em Let $g_n:\Omega \to Y$ be a sequence of strongly $\mu$-measurable functions
and let $g:\Omega \to Y$ be a function. If $\lim_{n\to \infty}\|g_{n}(\omega)- g(\omega)\|= 0$
for $\mu$-a.e. $\omega\in \Omega$, then $g$ is strongly $\mu$-measurable.}

Now, in order to prove the lemma in full generality, it suffices to repeat the argument of the {\em Particular case} bearing
in mind that each $f_n$ is strongly $\mu$-measurable, by~(i) and the previous {\sc Fact}. This finishes the proof.
\end{proof}

Our next lemma deals with some measurability properties of operators. Before the proof we need to recall some facts: 
\begin{enumerate}
\item[(a)] Given a compact topological space~$K$, it is known that every norm continuous function $f:K\to Y$ is universally strongly measurable, i.e., it
is strongly measurable with respect to any Radon probability on~$K$; in fact, the same holds for a weakly continuous function, see \cite[Proposition~4]{ari-alt}.
This implies that for every $T\in \mathcal{F}_{w^*}(X^*,Y)$ the restriction $T|_{B_{X^*}}$ is strongly $\mu$-measurable
for any $\mu\in P(B_{X^*})$.

\item[(b)] Given a probability space $(\Omega,\Sigma,\mu)$ and a function $f:\Omega\to Y$, Pettis' measurability theorem (see, e.g., \cite[p.~42, Theorem~2]{die-uhl-J})
states that $f$ is strongly $\mu$-measurable if (and only if) it has
$\mu$-essentially separable range and for every $y^*\in Y^*$ the composition $y^*\circ f:\Omega\to \mathbb{R}$ is $\mu$-measurable.
\end{enumerate}

\begin{lem}\label{lem:stronglymeasurableoperators}
Let $T\in \mathcal{L}(X^*,Y)$ and $\mu\in P(B_{X^*})$. Let us consider the following conditions:
\begin{enumerate}
\item[(i)] there is a sequence $(T_n)$ in $\mathcal{F}_{w^*}(X^*,Y)$ such that
$\|T_n(x^*)-T(x^*)\|\to 0$ for $\mu$-a.e. $x^*\in B_{X^*}$; 
\item[(ii)] $T|_{B_{X^*}}$ is strongly $\mu$-measurable;
\item[(iii)] $T|_{B_{X^*}}$ has $\mu$-essentially separable range.
\end{enumerate}
Then (i)$\impli$(ii)$\impli$(iii). If $X$ contains no subspace isomorphic to~$\ell_1$, then (ii) and (iii) are equivalent.
\end{lem}
\begin{proof} The implication (i)$\impli$(ii) follows from Lemma~\ref{lem:double} and comment~(a) above.
As we have already mentioned, (ii)$\impli$(iii) holds for arbitrary functions.

The fact that (iii)$\impli$(ii) if $X$ contains no subspace isomorphic to~$\ell_1$
is essentially contained in \cite[Example~2.4]{rod-san} and can be deduced as follows. For each 
$y^*\in Y^*$ we have $y^*\circ T\in X^{**}$ and so the non-containment of~$\ell_1$ 
implies that $y^*\circ T|_{B_{X^*}}$ is $\mu$-measurable,
according to the result of Haydon~\cite{hay-J} mentioned in the introduction (cf., \cite[Theorems~6.8 and~6.9]{van} or \cite[Corollary~4.19]{mus3}).
By Pettis' measurability theorem (see~comment (b) above), $T|_{B_{X^*}}$ is strongly $\mu$-measurable if (and only
if) it has $\mu$-essentially separable range.
\end{proof}

\subsection{The Bourgain property}

This terminology refers to a kind of ``controlled oscillation'' property invented by Bourgain~\cite{bou-J-2} for families of real-valued functions defined on a probability space. 
Riddle and Saab~\cite{rid-saa} used it to study certain aspects of the Pettis integral theory
of Banach space-valued functions (cf., \cite[Chapter~4]{mus3} and~\cite{cas-rod}).
Here we will need an extension of the Bourgain property
for families of functions taking values in a metric space. 

Given a probability space $(\Omega,\Sigma,\mu)$, we write $\Sigma^+:=\{A\in \Sigma:\mu(A)>0\}$
and for each $A\in \Sigma^+$ we write $\Sigma^+_A:=\{B\in \Sigma^+: B \sub A\}$.

\begin{defi}\label{defi:Bourgain}
Let $(\Omega,\Sigma,\mu)$ be a probability space, $(M,d)$ be a metric space, and $\mathcal{H} \sub M^\Omega$ be
a family of functions from~$\Omega$ to~$M$.
We say that $\mathcal{H}$ has the \emph{Bourgain property} with respect to~$\mu$ if for each $\epsilon>0$ and each $A \in \Sigma^+$ 
there is a finite set $\mathcal{B} \sub \Sigma^+_A$ such that 
$$
	\min_{B\in \mathcal{B}}\, {\rm osc}(h,B) \leq \epsilon
	\quad
	\text{for every $h \in \mathcal{H}$,}
$$
where ${\rm osc}(h,B):=\sup\{d(h(\omega),h(\omega')):\omega,\omega'\in B\}$
is the oscillation of~$h$ on~$B$.
\end{defi}

We will need the following striking result of Bourgain, see \cite[Theorem~11]{rid-saa}
(cf., \cite[Proposition~4.13]{mus3}). The original statement deals
with real-valued functions but the proof can be adapted straightforwardly to the case 
of functions with values in a metric space (and so we omit it).

\begin{theo}[Bourgain]\label{theo:Bourgain}
Let $(\Omega,\Sigma,\mu)$ be a probability space, $(M,d)$ be a metric space, $\mathcal{H} \sub M^\Omega$
be a family having the Bourgain property with respect to~$\mu$, and $g\in M^\Omega$. If $g$ belongs to the $\mathfrak{T}_{p}$-closure of~$\mathcal{H}$, 
then there is a sequence $(h_n)$ in~$\mathcal{H}$ such that 
$$
	d(h_n(\omega),g(\omega))\to 0 
	\quad\text{for $\mu$-a.e. $\omega\in \Omega$.}
$$
\end{theo}

Given a compact topological space~$K$, we denote by $C(K)$ the Banach space of all real-valued
continuous functions on~$K$, equipped with the supremum norm. 
It is known that {\em for a norm-bounded set $\mathcal{H}\sub C(K)$ the following statements are equivalent:
\begin{enumerate}
\item[(i)] $\mathcal{H}$ has the Bourgain property with respect to any Radon probability on~$K$;
\item[(ii)] $\mathcal{H}$ contains no sequence equivalent to the usual unit basis of~$\ell_1$.
\end{enumerate}}
Indeed, a proof of (ii)$\impli$(i) can be found, e.g., in \cite[Proposition~4.15]{mus3}. On the other hand, 
(i) ensures that every $g\in \mathbb{R}^K$ belonging to the $\mathfrak{T}_p$-closure of~$\mathcal{H}$
is $\mu$-measurable for any Radon probability~$\mu$ on~$K$, thanks to Theorem~\ref{theo:Bourgain};
this condition is equivalent to~(ii) (see, e.g., \cite[Theorem~3.11]{van}). 

The following lemma is an application of the equivalence (i)$\Leftrightarrow$(ii) above.

\begin{lem}\label{lem:Bourgain}
Suppose that $X$ contains no subspace isomorphic to~$\ell_1$ and that $Y$ is finite dimensional. Then
$$
	\{T|_{B_{X^*}}: \, T\in \mathcal{F}_{w^*}(X^*,Y), \, \|T\|\leq 1\} \sub Y^{B_{X^*}}
$$
has the Bourgain property with respect to any~$\mu\in P(B_{X^*})$.
\end{lem}

For the proof of Lemma~\ref{lem:Bourgain} we need
an elementary observation:

\begin{sublem}\label{lem:max}
Let $(\Omega,\Sigma,\mu)$ be a probability space, $(M,d)$ be a metric space, and 
$\mathcal{H}_1,\dots,\mathcal{H}_p \sub M^{\Omega}$
be a finite collection of families having the Bourgain property with respect to~$\mu$. Then for each $\epsilon>0$ and each $A\in \Sigma^+$ there is a finite set
$\mathcal{B} \sub \Sigma_A^+$ such that 
$$
	\min_{B\in \mathcal{B}} \, \max_{i=1,\dots,p} \, {\rm osc}(h_i,B)\leq \epsilon
	\quad \text{for every $(h_1,\dots,h_p) \in \prod_{i=1}^p \mathcal{H}_i$.}
$$
\end{sublem}
\begin{proof}
We proceed by induction on~$p$. The case $p=1$ is obvious. Suppose that $p>1$ and that the statement holds for~$p-1$ families. 
Fix $\epsilon>0$ and $A\in \Sigma^+$. By the inductive hypothesis, 
there is a finite set $\mathcal{C} \sub \Sigma_A^+$ such that
$$
	\min_{C\in \mathcal{C}} \, \max_{i=1,\dots,p-1} \, {\rm osc}(h_i,C)\leq \epsilon
	\quad
	\text{for every $(h_1,\dots,h_{p-1}) \in \prod_{i=1}^{p-1} \mathcal{H}_i$.}
$$
Since $\mathcal{H}_p$ has the Bourgain property with respect to~$\mu$,
for each $C\in \mathcal{C}$ there is a finite set $\mathcal{B}_C \sub \Sigma_C^+$ such that
$$
	\min_{B\in \mathcal{B}_C} {\rm osc}(h_p,B)\leq \epsilon \quad\text{for every  $h_p\in \mathcal{H}_p$.}
$$
It is clear that the set $\mathcal{B}:=\bigcup_{C\in \mathcal{C}}\mathcal{B}_C$ satisfies the required property.
\end{proof}

\begin{proof}[Proof of Lemma~\ref{lem:Bourgain}] The case $Y=\{0\}$ is obvious. Suppose that $Y$ has dimension $p\geq 1$ and
let $\{y_1,\dots,y_p\}$ be a basis of~$Y$ with biorthogonal functionals $\{y_1^*,\dots,y_p^*\} \sub Y^*$.
 
Fix $i\in \{1,\dots,p\}$. For each $T\in \mathcal{F}_{w^*}(X^*,Y)$ the composition $y_i^*\circ T\in X^{**}$ is $w^*$-continuous, hence
it belongs to~$X$. Bearing in mind the canonical isometric embedding of~$X$ into~$C(B_{X^*},w^*)$
and the fact that $X$ contains no subspace isomorphic to~$\ell_1$, it follows that 
$$
	\mathcal{H}_i:=\{y_i^*\circ T|_{B_{X^*}}: \, \, T\in \mathcal{F}_{w^*}(X^*,Y), \, \|T\|\leq 1\}
$$
is a norm-bounded subset of $C(B_{X^*},w^*)$ containing no sequence equivalent to the usual unit basis of~$\ell_1$. According to the comments preceding Lemma~\ref{lem:Bourgain},
$\mathcal{H}_i$ has the Bourgain property with respect to any $\mu\in P(B_{X^*})$. 

Set $\alpha:=\sum_{i=1}^p \|y_i\|$. Observe that for each function $h:B_{X^*}\to Y$ and each $B\sub B_{X^*}$ we have
\begin{multline}\label{eqn:norm-inequality}
	{\rm osc}(h,B)=\sup_{x_1^*,x_2^*\in B}\Big\|\sum_{i=1}^p \big(y_i^*(h(x_1^*))-y_i^*(h(x_2^*))\big) \, y_i\Big\| 
	\\ \leq 
	\sup_{x_1^*,x_2^*\in B}\, \alpha \cdot \max_{i=1,\dots,p} \, \big|y_i^*(h(x_1^*))-y_i^*(h(x_2^*))\big|
	\leq \alpha\cdot \max_{i=1,\dots,p}{\rm osc}(y_i^*\circ h,B).
\end{multline}
Write $\Sigma:={\rm Borel}(B_{X^*},w^*)$ and fix $\mu\in P(B_{X^*})$. 
By Sublemma~\ref{lem:max} applied to the families $\mathcal{H}_1,\dots,\mathcal{H}_p$, for each $\epsilon>0$ and 
each $A\in \Sigma^+$ there is a finite set $\mathcal{B}\sub \Sigma_A^+$ such that 
$$
	\min_{B\in \mathcal{B}} \, {\rm osc}(T|_{B_{X^*}},B)
	\stackrel{\eqref{eqn:norm-inequality}}{\leq}
	\alpha\cdot \min_{B\in \mathcal{B}} \, \max_{i=1,\dots,p}\, {\rm osc}(y_i^*\circ T|_{B_{X^*}},B)\leq \alpha\cdot \epsilon
$$
for every $T\in \mathcal{F}_{w^*}(X^*,Y)$ with $\|T\|\leq 1$.
This shows that 
$$
	\{T|_{B_{X^*}}: \, T\in \mathcal{F}_{w^*}(X^*,Y), \, \|T\|\leq 1\}
$$ 
has the Bourgain property with respect to~$\mu$.
\end{proof}

\subsection{Proof of Theorem~\ref{theo:Main}}

A key point in the argument will be the existence (under certain assumptions) of a suitable isometric embedding of 
$\mathcal{L}(X^*,Y)$ into the bidual of a Banach space, in such a way that 
the restriction of the $w^*$-topology to $\mathcal{L}(X^*,Y)$ coincides with the so-called weak operator topology.

\begin{defi}\label{dei:topologies}
Given two Banach spaces $E$ and~$F$, we denote by $\mathcal{L}(E,F)$ the Banach space of all operators from~$E$ to~$F$,
equipped with the operator norm. The {\em strong operator topology --SOT--} (resp., {\em weak operator topology --WOT--}) 
on $\mathcal{L}(E,F)$
is the locally convex topology for which the sets 
$$
	\{T\in \mathcal{L}(E,F): \, \|T(e)\|<\epsilon\} 
	\quad \text{where $e\in E$ and $\epsilon>0$}
$$ 
$$
	\mbox{(resp.,} \
	\{T\in \mathcal{L}(E,F): \, |\langle T(e),f^*\rangle|<\epsilon\} 
	\quad \text{where $e\in E$, $f^*\in F^*$ and $\epsilon>0$)}
$$
are a subbasis of open neighborhoods of~$0$. Therefore,
a net $(T_\alpha)$ in $\mathcal{L}(E,F)$ is SOT-convergent (resp., WOT-convergent) to~$T\in \mathcal{L}(E,F)$
if and only if $T_\alpha(e)\to T(e)$ in norm (resp., weakly) for every $e\in E$.

If $F=Z^*$ for a Banach space~$Z$, then the {\em weak$^*$ operator topology --W$^*$OT--} 
on $\mathcal{L}(E,F)$ is the locally convex topology for which the sets 
$$
	\{T\in \mathcal{L}(E,F): \, |\langle T(e),z\rangle|<\epsilon\}
	\quad\text{where $e\in E$, $z\in Z$ and $\epsilon>0$}
$$ 
are a subbasis of open neighborhoods of~$0$. Therefore, in this case a net $(T_\alpha)$ in $\mathcal{L}(E,F)$ is W$^*$OT-convergent 
to~$T\in \mathcal{L}(E,F)$ if and only if $T_\alpha(e)\to T(e)$ in the weak$^*$-topology for every $e\in E$. 
\end{defi}

\begin{lem}\label{lem:bidual}
Suppose that $Y$ is Asplund and that $Y^*$ has the approximation property. 
Let $Z:=X\hat{\otimes}_\epsilon Y$ be the injective tensor product of~$X$ and~$Y$. Then there is an isometric isomorphism
$\Phi: \mathcal{L}(X^*,Y^{**}) \to Z^{**}$ such that
\begin{enumerate}
\item[(i)] $Z$ is the norm-closure of $\Phi(\mathcal{F}_{w^*}(X^*,Y))$ and so
$$
	\Phi^{-1}(B_Z)=\overline{\{T\in \mathcal{F}_{w^*}(X^*,Y): \ \|T\|\leq 1\}}^{{\rm norm}};
$$
\item[(ii)] $\Phi$ is a $W^*OT$-to-$w^*$ homeomorphism on norm-bounded sets.
\end{enumerate}
\end{lem}
\begin{proof}
Let $X^* \hat{\otimes}_\pi Y^*$ be the projective tensor product of~$X^*$ and~$Y^*$. Then there is
an isometric isomorphism $\eta: \mathcal{L}(X^*,Y^{**})\to (X^* \hat{\otimes}_\pi Y^*)^*$ such that 
$$
	\langle \eta(T),x^*\otimes y^* \rangle=
	\langle T(x^*),y^* \rangle
$$
for every $T\in \mathcal{L}(X^*,Y^{**})$, $x^*\in X^*$ and $y^*\in Y^*$; see, e.g., \cite[Proposition~16.16]{fab-ultimo} 
(this does not require the additional assumptions on~$Y$). On the other hand, 
since $Y$ is Asplund and $Y^*$ has the approximation property, there is an isometric isomorphism
$\xi: X^* \hat{\otimes}_\pi Y^* \to Z^*$ such that 
$$
	\langle \xi(x^*\otimes y^*),x\otimes y\rangle=
	x^*(x)y^*(y) 
$$	
for every $x\in X$, $y\in Y$, $x^*\in X^*$ and $y^*\in Y^*$ (see, e.g., \cite[Theorem~16.40]{fab-ultimo}). 
Now, it is routine to check that $\Phi:=(\xi^{-1})^*\circ \eta$
satisfies the required properties.
\end{proof}

We are now ready to prove a particular case of Theorem~\ref{theo:Main} dealing with finite rank operators.

\begin{lem}\label{lem:fd}
Suppose that $X$ contains no subspace isomorphic to~$\ell_1$. If an operator $T\in \mathcal{L}(X^*,Y)$
has finite rank, then for each $\mu\in P(B_{X^*})$ 
there is a sequence $(T_n)$ in $\mathcal{F}_{w^*}(X^*,Y)$ such that
$$
	\|T_n(x^*)-T(x^*)\|\to 0
	\quad\text{for $\mu$-a.e. $x^*\in B_{X^*}$}
$$ 
and $\|T_n\| \leq \|T\|$ for every $n\in \mathbb{N}$.
\end{lem}
\begin{proof} Clearly, we can assume without loss of generality that $\|T\|=1$ and that $Y$ is finite dimensional,
hence W$^*$OT=SOT on $\mathcal{L}(X^*,Y^{**})=\mathcal{L}(X^*,Y)$.
Let $Z$ and~$\Phi$ be as in Lemma~\ref{lem:bidual}. By Goldstine's theorem,
$B_Z$ is $w^*$-dense in~$B_{Z^{**}}$ and so
\begin{equation}\label{eqn:b1}
	\overline{\Phi^{-1}(B_Z)}^{{\rm SOT}}
	=B_{\mathcal{L}(X^*,Y)}.
\end{equation}
On the other hand, the set $\Gamma:=\{S\in \mathcal{F}_{w^*}(X^*,Y): \, \|S\|\leq 1\}$ satisfies
\begin{equation}\label{eqn:b2}
	\Phi^{-1}(B_Z)=\overline{\Gamma}^{{\rm norm}}.
\end{equation}
Bearing in mind that $B_{\mathcal{L}(X^*,Y)}$ is SOT-closed and that the norm topology is finer than SOT, 
from equalities~\eqref{eqn:b1} and~\eqref{eqn:b2} we get
$$
	\overline{\Gamma}^{{\rm SOT}}=B_{\mathcal{L}(X^*,Y)}.
$$
Therefore
\begin{equation}\label{eqn:in}
	T|_{B_{X^*}} \in \overline{\{S|_{B_{X^*}}: \, S\in \Gamma\}}^{\mathfrak{T}_p} \sub 
	Y^{B_{X^*}}.
\end{equation}
On the other hand, Lemma~\ref{lem:Bourgain} ensures that 
$\{S|_{B_{X^*}}:S\in \Gamma\}\sub Y^{B_{X^*}}$ has the Bourgain property with respect to~$\mu$.
From \eqref{eqn:in} and Theorem~\ref{theo:Bourgain}
it follows that there is a sequence $(T_n)$ in $\Gamma$ such that $\|T_n(x^*)-T(x^*)\|\to 0$ 
for $\mu$-a.e. $x^*\in B_{X^*}$.
\end{proof}

Theorem~\ref{theo:Main} will be a consequence of the following result.

\begin{theo}\label{theo:nol1-ap}
Suppose that $X$ contains no subspace isomorphic to~$\ell_1$ and that $Y$ has the approximation property.
Let $T\in \mathcal{L}(X^*,Y)$ and $\mu\in P(B_{X^*})$.
If $T|_{B_{X^*}}$ is strongly $\mu$-measurable, then 
there is a sequence $(T_n)$ in $\mathcal{F}_{w^*}(X^*,Y)$ such that 
$$
	\|T_n(x^*)-T(x^*)\|\to 0 \quad\text{for $\mu$-a.e. $x^*\in B_{X^*}$.}
$$
If $Y$ has the $\lambda$-bounded approximation property for some $\lambda\geq 1$, then the sequence
can be chosen in such a way that $\|T_n\|\leq \lambda\|T\|$ for all $n\in \N$.
\end{theo}
\begin{proof} Fix $n\in \N$. Since $T|_{B_{X^*}}$ is strongly $\mu$-measurable, 
we can apply Lemma~\ref{lem:tight} to find $B_n \in {\rm Borel}(B_{X^*},w^*)$ such that $\mu(B_n)\geq 1-2^{-n}$
and $T(B_n)$ is relatively norm-compact. Since $Y$ has the approximation property, there is a finite rank operator $L_n:Y\to Y$ such that
\begin{equation}\label{eqn:aprox}
	\sup_{x^*\in B_n}\|L_{n}(T(x^*)) - T(x^*)\| \leq \frac{1}{n}.
\end{equation}
Now, Lemma~\ref{lem:fd} applied to the finite rank operator $L_n\circ T \in \mathcal{L}(X^*,Y)$ ensures the existence of 
a sequence $(S_{k,n})_{k\in \N}$ in $\mathcal{F}_{w^*}(X^*,Y)$
such that 
$$
	\lim_{k\to \infty}\|S_{k,n}(x^*)-L_n(T(x^*))\|= 0
	\quad\text{for $\mu$-a.e. $x^*\in B_{X^*}$}
$$
and $\|S_{k,n}\|\leq \|L_n\circ T\|\leq \|L_n\|\|T\|$ for all $k\in \N$.

Define 
$$
	B:=\bigcup_{n\in \N}\bigcap_{m\geq n}B_m \in {\rm Borel}(B_{X^*}w^*). 
$$
Then $\mu(B)=1$ and $\|L_{n}(T(x^*))-T(x^*)\|\to 0$ as $n\to \infty$ for every $x^*\in B$ (by~\eqref{eqn:aprox}). Finally, we can use Lemma~\ref{lem:double} to obtain
a sequence $(k_n)$ in~$\N$ such that the operators $T_n:=S_{k_n,n}$ satisfy the required property.

If in addition $Y$ has the $\lambda$-bounded approximation property for some $\lambda\geq 1$, then we can assume that $\|L_n\|\leq \lambda$ 
and so $\|T_n\|\leq \lambda\|T\|$ for all $n\in \N$.
\end{proof}

\begin{proof}[Proof of Theorem~\ref{theo:Main}]
By Theorem~\ref{theo:nol1-ap} it suffices to check that $T|_{B_{X^*}}$ is strongly $\mu$-measurable.
In case~(ii) this follows from Lemma~\ref{lem:stronglymeasurableoperators}.
As to case~(i), if $X$ is Asplund, then the identity function $i: (B_{X^*},w^*)\to X^*$
is strongly $\mu$-measurable, according to a result of Schwartz (see, e.g., \cite[Corollary~7.8.7(a)]{bou-J} or
\cite[Theorem~16.28]{fab-ultimo}). Therefore, $T|_{B_{X^*}}=T\circ i$ is strongly $\mu$-measurable as well.
\end{proof}

\subsection{Further results}

Recall that a Banach space is said to have the {\em metric approximation property}
if it has the $1$-bounded approximation property.

The following example shows that the conclusion of Theorems~\ref{theo:Main} and~\ref{theo:nol1-ap}
can fail for arbitrary operators if $X$ is not Asplund. 

\begin{exa}\label{exa:JT}
Let $JT$ be the James tree space. Then $JT$ is separable, it does not contain subspaces isomorphic to~$\ell_1$
and $JT^*$ is not separable. Moreover:
\begin{enumerate}
\item[(i)] The bidual $JT^{**}$ has the metric approximation property, see \cite[Corollary~4.20]{bra}. Therefore, 
$JT^*$ has the metric approximation property as well (see, e.g., \cite[p.~244, Corollary~9]{die-uhl-J}). 
\item[(ii)] Since $JT$ is not Asplund, there is some $\mu\in P(B_{JT^*})$
for which the identity map $(B_{JT^*},w^*)\to JT^*$ is not strongly $\mu$-measurable (see, e.g., \cite[Corollary~7.8.7(a)]{bou-J}).
Hence there is no sequence $(T_n)$ in $\mathcal{F}_{w^*}(JT^*,JT^*)$ such that $\|T_n(x^*)-x^*\|\to 0$ for $\mu$-a.e. $x^*\in B_{JT^*}$
(see Lemma~\ref{lem:stronglymeasurableoperators}).
\end{enumerate}
\end{exa}

It is known that the dual of an Asplund space has the metric approximation property 
if and only if it has the approximation property (see, e.g., \cite[Theorem~16.66]{fab-ultimo}).
Therefore, Theorem~\ref{theo:Main} applied to the identity operator on~$X^*$ yields the following:

\begin{cor}\label{cor:dualAsplund}
Suppose that $X$ is Asplund and that $X^*$ has the approximation property.
Let $\mu\in P(B_{X^*})$. Then: 
\begin{enumerate}
\item[(i)] there is a sequence $S_n\in \mathcal{F}_{w^*}(X^*,X^*)$
such that $\|S_n(x^*)-x^*\|\to 0$ for $\mu$-a.e. $x^*\in B_{X^*}$ and $\|S_n\|\leq 1$ for all $n\in \N$;
\item[(ii)] for each $T\in \mathcal{L}(X^*,Y)$, the sequence $T_n:=T\circ S_n\in \mathcal{F}_{w^*}(X^*,Y)$
satisfies $\|T_n(x^*)-T(x^*)\|\to 0$ for $\mu$-a.e. $x^*\in B_{X^*}$ and $\|T_n\|\leq \|T\|$ for all $n\in \N$.
\end{enumerate}
\end{cor}

By arguing as in the introduction when we showed that Theorem~\ref{theo:Main} implies Theorem~\ref{theo:MS}, we get
the following corollary. This result is due to Mercourakis and Stamati except for the norm inequalities, 
see \cite[Proposition~2.18]{mer-sta}.

\begin{cor}\label{cor:Mercourakis-Stamati}
Suppose that $X^*$ is separable and has the approximation property.
Then: 
\begin{enumerate}
\item[(i)] there is a sequence $S_n\in \mathcal{F}_{w^*}(X^*,X^*)$
such that $\|S_n(x^*)-x^*\|\to 0$ for every $x^*\in X^*$ and $\|S_n\|\leq 1$ for all $n\in \N$;
\item[(ii)] for each $T\in \mathcal{L}(X^*,Y)$, the sequence $T_n:=T\circ S_n\in \mathcal{F}_{w^*}(X^*,Y)$
satisfies $\|T_n(x^*)-T(x^*)\|\to 0$ for every $x^*\in X^*$ and $\|T_n\|\leq \|T\|$ for all $n\in \N$.
\end{enumerate}
\end{cor}

We finish the paper with a result in the spirit of Theorem~\ref{theo:MS}
but considering convergence with respect to the weak topology. It applies
to the example given in \cite[Theorem 2.30]{mer-sta}.
 
\begin{theo}\label{theo:WOT}
Suppose that $X$ is separable and does not contain subspaces isomorphic to~$\ell_1$, and that
$Y^*$ is separable and has the approximation property. Let $T\in \mathcal{L}(X^*,Y)$. Then
there is a sequence $(T_n)$ in $\mathcal{F}_{w^*}(X^*,Y)$ which WOT-converges to~$T$
such that $\|T_n\|\leq \|T\|$ for all $n\in \N$.
\end{theo}
\begin{proof}
Since $X$ contains no subspace isomorphic to~$\ell_1$ and $Y$ is Asplund, their injective tensor product
$Z:=X\hat{\otimes}_\epsilon Y$ contains no subspace isomorphic to~$\ell_1$, according to a result
by Rosenthal (see \cite[Corollary~4]{ros07}). Since $X$ and $Y$ are separable, so is~$Z$. Thus, by the Odell-Rosenthal theorem~\cite{ode-ros}
(cf., \cite[Theorem~4.1]{van}), $B_Z$ is $w^*$-sequentially dense in~$B_{Z^{**}}$.
Bearing in mind Lemma~\ref{lem:bidual}, we conclude that $\{S \in \mathcal{F}_{w^*}(X^*,Y):\, \|S\|\leq 1\}$ 
is W$^*$OT-sequentially dense in $B_{\mathcal{L}(X^*,Y^{**})}$. 

Therefore, for a given $T\in \mathcal{L}(X^*,Y)$,
there is a sequence $(T_n)$ in $\mathcal{F}_{w^*}(X^*,Y)$ which WOT-converges to~$T$
such that $\|T_n\|\leq \|T\|$ for all $n\in \N$.
\end{proof}

\subsection*{Acknowledgements}
Research partially supported by {\em Agencia Estatal de Investigaci\'{o}n} [MTM2017-86182-P, grant cofunded by ERDF, EU] 
and {\em Fundaci\'on S\'eneca} [20797/PI/18].

\bibliographystyle{amsplain}

\end{document}